\documentclass[12pt]{article}
\usepackage{amsmath, amsfonts, amsthm,amssymb, color, hyperref, enumerate, extarrows, cite}

 \usepackage[titletoc]{appendix}

\newtheorem{theorem}{Theorem}
\newtheorem{definition}[theorem]{Definition}
\newtheorem{cor}[theorem]{Corollary}
\newtheorem{lem}[theorem]{Lemma}
\newtheorem*{notation*}{Notation}
{\theoremstyle{definition}
\newtheorem{example}[theorem]{Example}}
\newtheorem{pro}[theorem]{Proposition}

\numberwithin{equation}{section}
{\theoremstyle{definition}
}
\usepackage[T1]{fontenc}

\hypersetup{
	colorlinks   = true,
	citecolor    = magenta}

\usepackage[titletoc]{appendix}

\begin{document}

\title{\vspace{-1.2cm} \bf Sufficient condition for compactness of the $\overline{\partial}$-Neumann operator using the Levi core\rm}

\author{John N. Treuer}
\date{}

\maketitle

\begin{abstract}

On a smooth, bounded pseudoconvex domain $\Omega$ in $\mathbb{C}^n$, to verify that Catlin's Property ($P$) holds for $b\Omega$, it suffices to check that it holds on the set of D'Angelo infinite type boundary points.  In this note, we consider the support of the Levi core, $S_{\mathfrak{C}(\mathcal{N})}$, a subset of the infinite type points, and show that Property ($P$) holds for $b\Omega$ if and only if it holds for $S_{\mathfrak{C}(\mathcal{N})}$.  Consequently, if Property ($P$) holds on $S_{\mathfrak{C}(\mathcal{N})}$, then the $\overline{\partial}$-Neumann operator $N_1$ is compact on $\Omega$.

\end{abstract}

\renewcommand{\thefootnote}{\fnsymbol{footnote}}
\footnotetext{\hspace*{-7mm} 
\begin{tabular}{@{}r@{}p{16.5cm}@{}}
& Keywords. $\overline{\partial}$-Neumann problem, compactness, Property ($P$), Levi core, holomorphic dimension zero\\
& Mathematics Subject Classification. Primary 32W05
\end{tabular}}

\section{Introduction}

Let $\Omega \subset \mathbb{C}^n$ be a $(C^{\infty})$-smooth, bounded pseudoconvex domain. Catlin introduced a potential theoretic condition called Property ($P$) and showed that if the boundary $b\Omega$ of $\Omega$ satisfies Property ($P$), then the $\overline{\partial}$-Neumann operator $N_1$ is compact \cite{Ca84}. Compactness has been a source of interest because it implies global regularity and it is believed that it can be more readily characterized than global regularity in terms of the boundary geometry of the domain.  We refer the reader to the monograph \cite{St10} for details on this point as well as to other applications of compactness.  Property ($P$) is useful as a sufficient condition for compactness because  for many types of domains, it is more tractable to show that Property ($P$) holds than to verify compactness directly.  Case in point, Catlin \cite{Ca84} proved that if $\Omega$ is additionally a (D'Angelo) finite type domain, then $N_1$ is compact because Property ($P$) holds.  Whether Property ($P$) is necessary for compactness remains an open question.  Fu and Straube \cite{FuSt01, FuSt02} showed that the two are equivalent on the locally convexifiable domains while Christ and Fu \cite{ChFu05} showed they are equivalent on Hartogs domains in $\mathbb{C}^2$.  See also \cite{Mc02, MuSt07, Zh21} for additional works that address the relationship between compactness and Property ($P$) or one of its variants.

Using the framework of Choquet theory, Sibony \cite{Si87} generalized Property ($P$) from boundaries of smooth, bounded pseudoconvex domains to compact sets in $\mathbb{C}^n$.  Through his and Catlin's work, the problem of determining if Property ($P$) holds on $b\Omega$ reduces to determining if Property ($P$) holds on the compact subset of infinite type boundary points.  In this note, we reduce the problem further from determining if Property ($P$) holds on the infinite type points to determining if Property ($P$) holds on a compact subset of the infinite type points called the support of the Levi core. 

Let $\mathcal{N}$ denote the Levi null distribution (in the sense of manifold theory), the set of pairs $(p, X)$ where $p \in b\Omega$ and $X = \sum_{k=1}^n a_k{\partial / \partial{z_k}}|_{p}$ is a Levi null vector at $p$.  Recently, Dall'Ara and Mongodi \cite{DaMo21} considered a (potentially transfinite) decreasing sequence of distributions $\mathcal{N}, \mathcal{N}^{1}, \mathcal{N}^2,\ldots$ where $\mathcal{N}^{\alpha + 1}$ consists of the Levi null vectors of $\mathcal{N}^{\alpha}$ that are tangent to the support of $\mathcal{N}^{\alpha}$.  If $\beta$ is a limit ordinal, then $\mathcal{N}^{\beta}$ is the intersection of the previous distributions in the sequence.  They showed that this sequence stabilizes in countably many steps; there is an ordinal $\gamma$ such that $\mathcal{N}^{\gamma} = \mathcal{N}^{\eta}$ for all $\eta \geq \gamma$, and they called $\mathcal{N}^{\gamma}$ the Levi core $\mathfrak{C}(\mathcal{N})$ of $b\Omega$.  They computed $\mathfrak{C}(\mathcal{N})$ for a number of different types of domains and CR-manifolds and gave applications of the Levi core to the Diederich-Forn\ae{}ss index, the D'Angelo class and regularity properties of the $\overline{\partial}$-Neumann operator.

Our main result examines the relationship between the support of the Levi core and Property ($P$).

\begin{theorem}\label{main theorem}
Let $\Omega \subset \mathbb{C}^n$ be a smooth, bounded pseudoconvex domain.  Property ($P$) holds on $b\Omega$ if and only if Property ($P$) holds on the support of the Levi core, $S_{\mathfrak{C}(\mathcal{N})}$.
\end{theorem}

In \cite{Bo88, Si87}, Boas and Sibony independently proved that for a smooth, bounded pseudoconvex domain \emph{if the Hausdorff two-dimensional measure of the infinite type points is zero, then Property ($P$) holds for $b\Omega$}.  As a corollary of Theorem \ref{main theorem}, we can give the following improvement to Boas and Sibony's result.

\begin{cor}\label{main corollary}
Let $\Omega$ be as above.  If the Hausdorff two-dimensional measure of $S_{\mathfrak{C}(\mathcal{N})}$ is zero, then Property ($P$) holds for $b\Omega$.
\end{cor}

In \cite{DaMo21}, Dall'Ara and Mongodi showed that for a smooth, bounded pseudoconvex domain \emph{if $\mathfrak{C}(\mathcal{N})$ is trivial (equivalently $S_{\mathfrak{C}(\mathcal{N})}$ is empty), then the $\overline{\partial}$-Neumann operator is exactly regular}.  In view of Theorem \ref{main theorem} and Corollary \ref{main corollary}, we can improve exact regularity to compactness.

\begin{cor}\label{corollary two}
Let $\Omega$ be as above.  If the Hausdorff two-dimensional measure of $S_{\mathfrak{C}(\mathcal{N})}$ is zero, then the $\overline{\partial}$-Neumann operator $N_1$ is compact.  
\end{cor}

We conclude the paper by providing an example of a Hartogs domain in $\mathbb{C}^2$ where the boundary satisfies Property ($P$) and the support of the Levi core has Hausdorff dimension three with positive three-dimensional Hausdorff measure.  From this example, we conclude that Corollary \ref{main corollary} is a sufficient but not a necessary condition for Property ($P$) to hold.

\section{Property ($P$) and the Levi core}

We begin by recalling the definition of Property ($P$).

\begin{definition}
Let $K$ be a compact set in $\mathbb{C}^n$.  $K$ satisfies Property ($P$) if for all $M > 0$ there is a neighborhood $U$ of $K$ and $\lambda_M \in C^2(U),\, 0 \leq \lambda_M \leq 1$, such that
$$
\sum_{j, k = 1}^n {\partial^2 \lambda_M \over \partial z_j \partial\overline{z}_k}(z) w_j\overline{w}_k \geq M|w|^2, \quad z \in U, \quad w \in \mathbb{C}^n.
$$
\end{definition}

The problem of showing Property ($P$) holds on $b\Omega$ can always be reduced in two ways.  First, if a compact set $K$ can be written as a countable union of compact sets satisfying Property ($P$), then Property ($P$) holds on $K$.  Consequently, deciding if Property ($P$) holds on $b\Omega$ reduces to deciding if Property ($P$) holds on the compact subsets where obstructions to Property ($P$) might occur.  Secondly, points of D'Angelo finite type are always benign for Property ($P$).  A point $p \in b\Omega$ is said to be of (D'Angelo) finite type if the order of contact with $b\Omega$ of complex analytic varieties passing through $p$ is bounded above.  Otherwise $p$ is said to be of infinite type, \cite{DA82}.  By Catlin's work \cite{Ca84, Ca84-2}, one can find a suitable covering of the finite type points by submanifolds that satisfy the hypotheses of the following useful proposition.  We use the formulation in Straube's monograph \cite[Proposition 4.15]{St10}.  See also \cite[Proposition 1.12]{Si87}, where it was originally stated explicitly.

\begin{pro}\label{holomorphic dimension zero proposition} Let $\Omega \subset \mathbb{C}^n$ be a smooth, bounded pseudoconvex domain and $S$ a smooth submanifold of $b\Omega$.  If 
\begin{equation}\label{holomorphic dimension zero condition}
\dim(\mathbb{C}T_pS \cap \mathcal{N}_p) = 0, \quad p \in S,
\end{equation}
then any compact subset of $S$ satisfies Property ($P$).

\end{pro}

Submanifolds that satisfy the hypotheses of Proposition \ref{holomorphic dimension zero proposition} are said to have {\bf holomorphic dimension zero}.  One of the consequences of Catlin's work is that the open set of points of finite type of $b\Omega$ can be covered by submanifolds of holomorphic dimension zero. It follows that $b\Omega$ then has a covering by compact sets all of which except possibly one, the compact set of infinite type points, satisfy Property ($P$).  In this way, we conclude that Property ($P$) holds on $b\Omega$ if and only if it holds on the set of infinite type points.  In the proof of Theorem \ref{main theorem}, instead of isolating the infinite type points, we will give a different covering of $b\Omega$ by using the construction of the Levi core and Proposition \ref{holomorphic dimension zero proposition}.

Let $\mathcal{D} = \{\mathcal{D}_p\}_{p \in b\Omega}$ be a complex distribution of a smooth manifold $M$ where $\mathcal{D}_p$ is a linear subspace over $\mathbb{C}$ of $\mathbb{C}T_pM$.  We define the support of $\mathcal{D}$ to be the set $S_{\mathcal{D}}$ of all $p \in M$ such that $\mathcal{D}_p$ is positive dimensional.  Although $S_{\mathcal{D}}$ may not have a smooth manifold structure, at each point $p \in M$, we can consider its (complexified) $C^{\infty}$-Zariski tangent space.

\begin{definition}\cite[Definition 2.5]{DaMo21}
Let $A \subset M$.  The $C^{\infty}$-Zariski tangent space at $p \in M$ is
$$
T_pA = \{X_p: X_p \in T_pM,\, X_pf = 0\hbox{ for all } f \in C^{\infty}(M)\hbox{ with } f|_A \equiv 0\}.
$$
The complexified tangent space $\mathbb{C}T_pA$ is the complexification of $T_pA$.
\end{definition}

The tangent space $TA = \sqcup_{p}T_pA$ (resp. $\mathbb{C}TA$) is always closed in $TM$ (resp. $\mathbb{C}TM$).  In the case where $A$ is an embedded submanifold, $T_pA$ agrees with the $C^{\infty}$-tangent space of $A$ at $p$, and for any set $A$ and any point $p \in A$, one can approximate $A$ by an embedded submanifold in the following sense.

\begin{lem}\cite[Proposition 2.6c]{DaMo21}\label{Proposition 2.6c of DaMo21}
Let $p$ be in a manifold $M$ containing a set $A$.  There is a real embedded submanifold $V \subset M$ locally containing $A$ such that $T_pA = T_pV$.
\end{lem}

By complexifying both sides, it follows that $\mathbb{C}T_pA = \mathbb{C}T_pV$.  Given a distribution $\mathcal{D}$, Dall'Ara and Mongodi considered the following transfinite recursively defined sequence of derived distributions $\{\mathcal{D}^{\alpha}\}$ indexed over the ordinals defined by
\begin{equation}\label{derived distributions definition}
\mathcal{D}_p^{\alpha} = \begin{cases}
\mathcal{D}_p & \alpha = 0
\\
\mathcal{D}_p^{\alpha - 1} \cap \mathbb{C}T_pS_{\mathcal{D}^{\alpha - 1}} & \alpha \hbox{ is a successor ordinal}
\\
\bigcap_{\beta < \alpha}\mathcal{D}_p^{\beta} & \alpha \hbox{ is a limit ordinal.}
\end{cases}
\end{equation}

As shown in \cite{DaMo21}, when $\mathcal{D}$ is closed in the topology of $\mathbb{C}TM$, we get a number of desirable properties: $\mathcal{D}^{\alpha}$ and $S_{\mathcal{D}^{\alpha}}$ are closed in their respective topologies for all $\alpha$, the dimension of $\mathcal{D}_p$ is an upper semicontinuous function of $p \in M$, and the sequence $\{\mathcal{D}^{\alpha}\}$ stabilizes after countably many ordinals.  That is, there is a countable ordinal $\gamma$ such that $\mathcal{D}^{\gamma} = \mathcal{D}^{\eta}$ for all $\eta \geq \gamma$.  We call $\mathcal{D}^{\gamma}$ the core $\mathfrak{C}(\mathcal{D})$ of the distribution $\mathcal{D}$.  In this note, we will focus on the distribution $\mathcal{D} = \mathcal{N}$ the Levi null distribution; its core $\mathfrak{C}(\mathcal{N})$ will be referred to as the Levi core. The support of the core $S_{\mathfrak{C}(\mathcal{N})}$ consists entirely of infinite type points. We begin the original material of this note by reformulating \eqref{derived distributions definition}.

\begin{lem}\label{rewriting the definition of derived distributions}
Let $\Omega$ be a smooth, bounded pseudoconvex domain and $p \in b\Omega$.  For any countable ordinal $\alpha$,
\begin{equation}\label{More useful characterization of of Nalpha}
\mathcal{N}_p^{\alpha} = \begin{cases}
\mathcal{N}_p & \alpha = 0
\\
\mathbb{C}T_pS_{\mathcal{N}^{\alpha - 1}} \cap \mathcal{N}_p & \alpha \hbox{ successor ordinal} 
\\
\bigcap\limits_{\beta < \alpha} \mathbb{C}T_pS_{\mathcal{N}^{\beta}} \cap \mathcal{N}_p & \alpha \hbox{ limit ordinal}
\end{cases}
\end{equation}
\end{lem}
\begin{proof}
We prove this claim using transfinite induction.  The base case follows by definition.  Assume that the claim holds for $\alpha$.  We will show it is true for its successor.  If $\alpha$ is a successor ordinal, then
\begin{eqnarray*}
\mathcal{N}_p^{\alpha + 1} &=& \mathbb{C}T_pS_{\mathcal{N}^{\alpha}} \cap \mathcal{N}_p^{\alpha}
\\
&=& \mathbb{C}T_pS_{\mathcal{N}^{\alpha}} \cap \mathbb{C}T_pS_{\mathcal{N}^{\alpha - 1}} \cap \mathcal{N}_p
\\
&=& \mathbb{C}T_pS_{\mathcal{N}^{\alpha}} \cap \mathcal{N}_p.
\end{eqnarray*}

If $\alpha$ is a limit ordinal, then by a similar argument, $\mathcal{N}_p^{\alpha + 1} = \mathbb{C}T_pS_{\mathcal{N}^{\alpha}} \cap \mathcal{N}_p$. Finally, suppose that $\alpha$ is a limit ordinal and the claim holds for all $\beta < \alpha$. Let $A = \{\beta: \beta < \alpha \hbox{ is a successor ordinal}\}$.  Since $\{\mathcal{N}^{\gamma}_p\}$ is a decreasing transfinite sequence,
\begin{eqnarray*}
\mathcal{N}_p^{\alpha} = \bigcap\limits_{\beta \in A} \mathcal{N}_p^{\beta}
= \bigcap\limits_{\beta \in A} \mathbb{C}T_pS_{\mathcal{N}^{\beta - 1}} \cap \mathcal{N}_p
= \bigcap\limits_{\beta < \alpha} \mathbb{C}T_pS_{\mathcal{N}^{\beta}} \cap \mathcal{N}_p.
\end{eqnarray*}

\end{proof}

Using this lemma, we can show that points outside of the support of the Levi core are locally contained in submanifolds of holomorphic dimension zero.  Below $\mathbb{B}(p, r)$ will denote the ball centered at $p$ of radius $r$.

\begin{lem} \label{Embedded Submanifold Holomorhpic Dimension Lemma}
 Let $\Omega$ and $p$ be as in the previous lemma.  Suppose $V$ is a real embedded submanifold in $b\Omega$ such that $$
 \mathbb{C}T_pV \cap \mathcal{N}_p^{\alpha} = \{0\} \hbox{ and } \mathbb{C}T_pV = \mathbb{C}T_pS_{\mathcal{N}^{\alpha}}
 $$ for some countable ordinal $\alpha$.  Then there exists an $r > 0$ such that $\mathbb{B}(p, r) \cap V$ has holomorphic dimension zero.
\end{lem}

\begin{proof}
If $\alpha = 0$, then the first hypothesis reduces to $\mathbb{C}T_pV \cap \mathcal{N}_p = \{0\}.$ If $\alpha$ is a successor ordinal, then by Lemma \ref{rewriting the definition of derived distributions},
$$
\{0\} = \mathbb{C}T_pV \cap \mathcal{N}_p^{\alpha} = \mathbb{C}T_pS_{\mathcal{N}^{\alpha}} \cap \mathbb{C}T_pS_{\mathcal{N}^{\alpha-1}} \cap \mathcal{N}_p = \mathbb{C}T_pV \cap \mathcal{N}_p.
$$
If $\alpha$ is a limit ordinal, then by similarly using Lemma \ref{rewriting the definition of derived distributions}, $\{0\} = \mathbb{C}T_pV \cap \mathcal{N}_p.$
Since $W_p = \mathbb{C}T_pV \cap \mathcal{N}_p$ defines a closed distribution when $p$ varies over $b\Omega$, the dimension of $W_p$ will be an upper semicontinuous function of $p \in b\Omega$.  Therefore, for sufficiently small $r > 0$,\, $\mathbb{B}(p, r) \cap V$ has holomorphic dimension zero.
\end{proof}

\begin{lem}\label{intersection support equals support at the limit ordinal}
Let $\Omega$ be as above.  If $\alpha$ is a limit ordinal, then 
$$
S_{\mathcal{N}^{\alpha}} = \bigcap_{\gamma < \alpha}S_{\mathcal{N}^{\gamma}}.
$$
\end{lem}
\begin{proof}
Since $\mathcal{N}^{\alpha} \subset \mathcal{N}^{\gamma}$ for all $\gamma < \alpha$, $S_{\mathcal{N}^{\alpha}} \subset \cap_{\gamma < \alpha}S_{\mathcal{N}^{\gamma}}$.  For the reverse containment, without loss of generality, let $z \in \cap_{\gamma < \alpha} S_{\mathcal{N}^{\gamma}}$ and consider
$$
\mathcal{UN}_z^{\gamma} = \mathcal{N}^{\gamma}_z \cap \{X \in \mathbb{C}T_zb\Omega: \|X\| = 1\}, \quad \gamma < \alpha.
$$
Then $\{\mathcal{UN}_z^{\gamma}\}_{\gamma < \alpha}$ is a decreasing transfinite sequence of non-empty compact sets in $\mathbb{C}T_zb\Omega$.  Thus, there exists 
$$
X_{\infty} \in \bigcap\limits_{\gamma < \alpha} \mathcal{UN}_z^{\gamma} \subset \mathcal{N}^{\alpha}_z.  
$$
Thus, $z \in S_{\mathcal{N}^{\alpha}}$.  The proof is complete.
\end{proof}

\begin{lem}\label{second transfinite induction lemma}
Let $\Omega$ be as above.  For any ordinal $\alpha \geq 1$,
\begin{equation}\label{Partition of weakly pseduconvex points}
S_{\mathcal{N}} = \bigcup_{0 \leq \gamma < \alpha} \left(S_{\mathcal{N}^{\gamma}}\setminus S_{\mathcal{N}^{\gamma + 1}}\right) \cup S_{\mathcal{N}^{\alpha}}.
\end{equation}
\end{lem}
\begin{proof}
We argue by transfinite induction on $\alpha$.  We omit the base case and successor ordinal cases and only show the limit ordinal case.  Let $\alpha$ be a limit ordinal and suppose
\begin{equation}\label{Limit ordinals inductive hypothesis for Lemma 11}
S_{\mathcal{N}} = \bigcup_{0 \leq \gamma < \beta} \left(S_{\mathcal{N}^{\gamma}}\setminus S_{\mathcal{N}^{\gamma + 1}}\right) \cup S_{\mathcal{N}^{\beta}}, \quad \beta < \alpha.
\end{equation}
It is clear the right hand side of \eqref{Partition of weakly pseduconvex points} is contained in the left hand side.  We show the other containment.  Suppose $z \in S_{\mathcal{N}}$ but $z \not\in \cup_{0 \leq \gamma < \alpha} (S_{\mathcal{N}^{\gamma}} \setminus S_{\mathcal{N}^{\gamma + 1}})$.  Then $z \not\in \cup_{0 \leq \gamma < \beta} (S_{\mathcal{N}^{\gamma}} \setminus S_{\mathcal{N}^{\gamma + 1}})$ for all $\beta < \alpha$.  By \eqref{Limit ordinals inductive hypothesis for Lemma 11}, $z \in S_{\mathcal{N}^{\beta}}$ for all $\beta < \alpha$.  By Lemma \ref{intersection support equals support at the limit ordinal}, $z \in S_{\mathcal{N}^{\alpha}}$.  The proof is complete.

\end{proof}

\section{The main results and example}

\begin{proof}[Proof of Theorem \ref{main theorem}]
Define $K_{-1}$ to be the strongly pseudoconvex boundary points.  Let $A$ be the set of all $\alpha \geq 0$ such that $\mathcal{N}^{\alpha} \neq \mathcal{N}^{\alpha + 1}$. For $\alpha \in A$, define
$
K_\alpha = S_{\mathcal{N}^{\alpha}} \setminus S_{\mathcal{N}^{\alpha + 1}}.
$
By Lemma \ref{second transfinite induction lemma}, the boundary of $\Omega$ has a partition
\begin{equation}\label{partition}
b\Omega = S_{\mathfrak{C}(\mathcal{N})} \cup K_{-1} \cup \bigcup_{\alpha \in A} K_{\alpha}.
\end{equation}

We claim that each $K_{\alpha},\, \alpha \geq -1$, can be covered by countably many compact sets satisfying Property ($P$). First, consider the $\alpha \geq 0$ cases. By Lemma \ref{Proposition 2.6c of DaMo21}, in a small neighborhood of any $p \in K_{\alpha}$ there is a real embedded submanifold $V$ of $b\Omega$ such that for $r_1$ sufficiently small
$$
V \supset \left(S_{\mathcal{N}^{\alpha}} \setminus S_{\mathcal{N}^{\alpha + 1}}\right) \cap \mathbb{B}(p, r_1) = S_{\mathcal{N}^{\alpha}} \cap \mathbb{B}(p, r_1)
$$
and
$$
\mathbb{C}T_p V = \mathbb{C}T_p (S_{\mathcal{N}^{\alpha}} \setminus S_{\mathcal{N}^{\alpha + 1}}) = \mathbb{C}T_p S_{\mathcal{N}^{\alpha}}.
$$
Since $p \in K_{\alpha}$,
$$
\mathbb{C} T_pV \cap \mathcal{N}_p^{\alpha} = \mathbb{C}T_pS_{\mathcal{N}^{\alpha}} \cap \mathcal{N}_p^{\alpha} = \mathcal{N}_p^{\alpha + 1} = \{0\}.
$$
By Lemma \ref{Embedded Submanifold Holomorhpic Dimension Lemma}, there exists $r_2$ with $r_1 > r_2 > 0$ such that $\mathbb{B}(p, r_2) \cap V$ has holomorphic dimension zero.  By Proposition \ref{holomorphic dimension zero proposition}, any compact subset of $\mathbb{B}(p, r_2) \cap V$ satisfies Property ($P$).  

Let $V_p \subset \mathbb{B}(p, r_2) \cap V$ be a homeomorphic image of a closed ball in $\mathbb{R}^{\dim(V)}$ with $p \in Int(V_p)$.  Since $\mathbb{B}(p, r_2) \cap V$ locally contains $K_{\alpha}$ near $p$, for sufficiently small $r_p > 0$, 
$$
\mathbb{B}(p, r_{p}) \cap K_{\alpha} \subset Int(V_p).
$$
Since every open covering of $K_{\alpha}$ will have a countable subcovering, there is a countable index set $I$ and points $\{p_j\} \subset K_{\alpha}$  such that
$$
K_{\alpha} = \bigcup\limits_{j \in I} \mathbb{B}(p_j, r_{p_j}) \cap K_{\alpha} \subset \bigcup_{j \in I}V_{p_j}.
$$
Since each $V_{p_j}$ satisfies Property ($P$), the claim is proved for $\alpha \geq 0$.  For the $\alpha = -1$ case, $K_{-1}$ is open in $b\Omega$ and every embedded submanifold of it has holomorphic dimension zero. Thus, $K_{-1}$ can be covered by countably many closed balls satisfying Property ($P$).  The claim is proved.  By \eqref{partition}, $b\Omega$ satisfies Property ($P$) if and only if $S_{\mathfrak{C}(\mathcal{N})}$ satisfies Property ($P$).

\end{proof}

\begin{proof}[Proof of Corollary \ref{main corollary}]
In \cite[Remarque p.310]{Si87}, using the language of $B$-regularity and the Jensen boundary, Sibony showed that if $b\Omega = S_1 \cup S_2$, $S_2$ is compact with Hausdorff two-dimensional measure zero, and Property ($P$) holds on all compact subsets of $S_1$, then Property ($P$) holds on $b\Omega$.  In \eqref{partition}, since Property ($P$) holds on compact subsets of $K_{\alpha}$ for $\alpha \geq -1$, if the Hausdorff two-dimensional measure of $S_{\mathfrak{C}(\mathcal{N})}$ equals zero, then Property ($P$) holds on $b\Omega$.
\end{proof}

\begin{proof}[Proof of Corollary \ref{corollary two}]
By the previous corollary, Property ($P$) holds on $b\Omega$. Therefore the $\overline{\partial}$-Neumann operator $N_1$ is compact.
\end{proof}

 In \cite[Example 1]{Bo88}, Boas gave an example of a Hartogs domain in $\mathbb{C}^2$ satisfying Property ($P$) in which the projection of the infinite type boundary points onto the complex plane forms a square Cantor set.  We modify the construction to give an example of a domain satisfying Property ($P$) in which the Hausdorff dimension of $S_{\mathfrak{C}(\mathcal{N})}$ is three and has positive three-dimensional Hausdorff measure.  In particular, what is new here is the calculation of the Levi core.
 
Let $\Omega$ be a complete Hartogs domain in $\mathbb{C}^2$ over the unit disk $\mathbb{D} \subset \mathbb{C}$
\begin{equation}\label{Hartogs domain prototype}
\Omega = \left\{(z, w):\, z\in \mathbb{D},\, |w|^2 < e^{-\phi(z)}\right\}.
\end{equation}
Given a compact subset $K \subset \{z:\, |z| < {1 \over 2}\}$, one can define $\phi$ so that $\Omega$ is smooth, bounded, pseudoconvex and 
\begin{equation}\label{formulas 1 for null sets in Hartogs domains}
S_{\mathcal{N}} = \left\{\left(z, e^{-{\phi(z) \over 2} + i \theta}\right):\, z \in K,\, \theta \in [0, 2\pi]\right\}, 
\end{equation}
and
\begin{equation}
\mathcal{N}_{(z, w)} = sp_{\mathbb{C}}\left\{\partial_z + e^{-{\phi(z) \over 2} + i\theta}\phi_z(z)\partial_w\right\}, \quad (z, w) \in S_{\mathcal{N}}.
\end{equation}
For such domains, Property ($P$) holds on $b\Omega$ if and only if $K$ has empty fine interior. See \cite[Section 4.7]{St10} for more details on the characterization of Property ($P$) for the complete Hartogs domains.  

\begin{example}
Consider a (fat) Cantor set of positive one-dimensional measure on the real line.  Let $K$ be the product of this set with itself, but rescaled so that $K \subset \mathbb{D}(0, 1/2)$.  Since $K$ has empty fine interior, the domain $\Omega$ given by \eqref{Hartogs domain prototype} satisfies Property ($P$).  Due to the rotational symmetry of $\Omega$, by \eqref{formulas 1 for null sets in Hartogs domains}, $S_{\mathcal{N}}$ has Hausdorff dimension three and positive three-dimensional Hausdorff measure. It remains to show that $\mathcal{N} = \mathfrak{C}(\mathcal{N})$, which implies that $S_{\mathcal{N}} = S_{\mathfrak{C}(\mathcal{N})}$.  We compute $\mathbb{C}TS_{\mathcal{N}}$.

Fix
$$
p = (z_0, e^{-{\phi(z_0) \over 2} + i\theta}) \hbox{ with } z_0 = x_0 + iy_0 \in K,
$$ 
and let $f$ be a smooth function on $b\Omega$ that vanishes on $S_{\mathcal{N}}$.  Define $h_i:\mathbb{R} \to b\Omega$, $i = 1, 2, 3$, by
 $$
 h_1(t) = (z_0, e^{-{\phi(z_0) \over 2} + it}), \quad h_2(t) = (x_0 + it, e^{-{\phi(x_0 + it) \over 2} + i\theta}), \quad h_3(t) = (t + iy_0, e^{-{\phi(t + iy_0) \over 2} + i\theta}).
 $$
Since $f\circ h_1 \equiv 0$, taking the derivative with respect to $t$ and evaluating at $t = \theta$ yields
$$
i{\partial f \over \partial w}(p)e^{-{\phi(z_0) \over 2} + i\theta} - i{\partial f \over \partial \overline{w}}(p)e^{-{\phi(z_0) \over 2} - i\theta}= 0
$$
Thus,
\begin{equation}\label{first tangent vector}
X_p = ie^{-{\phi(z_0) \over 2} + i\theta}\partial_w|_{p} - ie^{-{\phi(z_0) \over 2} - i\theta}\partial_{\overline{w}}|_{p}
\end{equation}
is in $\mathbb{C}T_pS_{\mathcal{N}}$.  Since $K$ is the product of perfect sets, there are sequences $(x_n), (y_n)$ approaching $x_0$ and $y_0$ respectively with $(x_n, y_0), (x_0, y_n) \in K$.  Then
$$
h_2(y_n), h_3(x_n) \in S_{\mathcal{N}}, \quad  n \in \mathbb{N}.
$$
It follows that 
\begin{eqnarray*}
0 = {d \over dt}(f \circ h_2)(y_0) = if_z(p) - if_{\overline{z}}(p) &+& f_{w}(p)e^{-{\phi(z_0) \over 2} + i\theta}\left(-{i \over 2}\phi_z(z_0) + {i \over 2}\phi_{\overline{z}}(z_0) \right) 
\\
&+& f_{\overline{w}}(p)e^{-{\phi(z_0) \over 2} - i\theta}\left(-{i \over 2}\phi_{z}(z_0) + {i \over 2}\phi_{\overline{z}}(z_0) \right)
\end{eqnarray*}
and
\begin{eqnarray*}
0 = {d \over dt}(f \circ h_3)(x_0) = f_z(p) + f_{\overline{z}}(p) &+& f_{w}(p)e^{-{\phi(z_0) \over 2} + i\theta}\left(-{1 \over 2}\phi_z(z_0) - {1 \over 2}\phi_{\overline{z}}(z_0) \right) 
\\
&+& f_{\overline{w}}(p)e^{-{\phi(z_0) \over 2} - i\theta}\left(-{1 \over 2}\phi_{z}(z_0) - {1 \over 2}\phi_{\overline{z}}(z_0) \right).
\end{eqnarray*}
Let $z = x + iy$.  The vectors
$$
Y_p = {\partial_{y}|_p} + e^{-{\phi(z_0) \over 2} + i\theta}\left(-{i \over 2}\phi_z(z_0) + {i \over 2}\phi_{\overline{z}}(z_0) \right)\partial_{w}|_{p} + e^{-{\phi(z_0) \over 2} - i\theta}\left(-{i \over 2}\phi_{z}(z_0) + {i \over 2}\phi_{\overline{z}}(z_0) \right)\partial_{\overline{w}}|_{p},
$$
$$
Z_p = \partial_{x}|_p + e^{-{\phi(z_0) \over 2} + i\theta}\left(-{1 \over 2}\phi_z(z_0) - {1 \over 2}\phi_{\overline{z}}(z_0) \right)\partial_{w}|_{p} + e^{-{\phi(z_0) \over 2} - i\theta}\left(-{1 \over 2}\phi_{z}(z_0) - {1 \over 2}\phi_{\overline{z}}(z_0) \right)\partial_{\overline{w}}|_{p}
$$
also belong to $\mathbb{C}T_pS_{\mathcal{N}}$.  The three vectors $X_p, Y_p, Z_p$ are linearly independent. By dimensional reasons, $\mathbb{C}T_pS_{\mathcal{N}} = \mathbb{C}T_pb\Omega$.  Thus, $\mathcal{N}_p^{1} = \mathcal{N}_p$.  If $p \not\in S_{\mathcal{N}}$, then trivially $\mathcal{N}_p^{1} = \mathcal{N}_p$.  It follows that $\mathcal{N} = \mathfrak{C}(\mathcal{N})$.
\end{example}

 \subsection*{Funding}

{\fontsize{11.5}{10}\selectfont

The research of the author is supported by an AMS-Simons travel grant.
}

  \subsection*{Acknowledgements}
{\fontsize{11.5}{10}\selectfont
The author would like to thank Emil Straube for giving valuable advice and feedback during the preparation of this note.  He would also like to thank Harold Boas, Tanuj Gupta and Shreedhar Bhat for useful conversations.  }

\bibliographystyle{alphaspecial}

\begin{thebibliography}{HD}


{\fontsize{11}{11.0}\selectfont


\bibitem{Bo88} Boas, Harold, Small sets of infinite type are benign for the $\overline{\partial}$-Neumann Problem, \emph{Proc. Amer. Math. Soc.} {\bf 103} (1988), 569-578.

\bibitem{Ca84-2}Catlin, David, Boundary Invariants of Pseudoconvex Domains, \emph{Ann. Math.
(2)}, {\bf 120} (1984), 529-586.

\bibitem{Ca84} Catlin, David, Global regularity of the $\overline{\partial}$-Neumann problem, In \emph{Complex Analysis of Several Variables} (ed. by Y.-T.~Siu). Proc. Sympos. Pure Math. 41, Amer. Math. Soc., Providence 1984, 39-49.

\bibitem{ChFu05} Christ, Michael and Fu, Siqi, Compactness in the $\overline{\partial}$-Neumann problem, magnetic Schr\"odinger operators, and the Aharonov-Bohm effect, \emph{Adv. Math.} {\bf 197} Nr. 1 (2005), 1-40.

\bibitem{DaMo21} Dall'Ara, Gian Maria and Mongodi, Samuele, The Core of the Levi Distribution, preprint 2021, arXiv:2109.04763.

\bibitem{DA82} D'Angelo, John, Real hypersurfaces, orders of contact, and applications, \emph{Ann. Math. (2)} {\bf 115} (1982), 615-637.

\bibitem{FuSt01} Fu, Siqi and Straube, Emil J., Compactness in the $\overline{\partial}$-Neumann problem. In \emph{Complex Analysis and Geometry} (ed. by J. McNeal). Ohio State Univ. Math. Res. Inst. Publ. 9, de Gruyter, Berlin 2001, 141-160.

\bibitem{FuSt02} Fu, Siqi and Straube, Emil J., Semi-classical analysis of Schr\"odinger operators and compactness in the $\overline{\partial}$-Neumann problem, \emph{J. Math. Anal. Appl.} {\bf 271} (2002), 267-282.

\bibitem{Mc02} McNeal, Jeffery D., A sufficient condition for compactness of the $\overline{\partial}$-Neumann operator, \emph{J. Funct. Anal.} {\bf 195} (2002), Nr. 1, 190-205.

\bibitem{MuSt07} Munasinghe, Samangi and Straube, Emil J., Complex tangential flows and compactness of the $\overline{\partial}$-Neumann operator, \emph{Pacific J. Math.} {\bf 232} (2007), 343-354.


\bibitem{Si87} Sibony, Nessim, Une classe de domaines pseudoconvexes, \emph{Duke Math. J.} {\bf 55} (1987), Nr. 2, 299-319.

\bibitem{St10} Straube, Emil J., \emph{Lectures on the $\mathcal{L}^{2}$-Sobolev theory of the $\overline{\partial}$-Neumann problem}, ESI Lectures in Mathematics and Physics, European Mathematical Society (EMS), Z\"{u}rich 2010.

\bibitem{Zh21} Zhang, Yue, Sufficient conditions for compactness of the $\overline{\partial}$-Neumann operator on high level forms, \emph{Pacific J. Math.} {\bf 313} (2021), 213-249. 

}


\end{thebibliography}

\fontsize{11}{9}\selectfont

\vspace{0.5cm}

\noindent jtreuer@tamu.edu;

 \vspace{0.2 cm}

\noindent Department of Mathematics, Texas A\&M University, College Station, TX 77843-3368, USA

\end{document}